\documentclass[12pt]{amsart}
\usepackage{amssymb}
\usepackage{amsmath, amscd}
\usepackage{amsthm}
\usepackage[table,xcdraw]{xcolor}
\usepackage[colorlinks=true, linkcolor=blue,urlcolor=blue]{hyperref}

\newtheorem{theorem}{Theorem}[section]

\newtheorem{proposition}[theorem]{Proposition}
\newtheorem{lemma}[theorem]{Lemma}

\newtheorem{corollary}[theorem]{Corollary}
\theoremstyle{definition}

\newtheorem{example}[theorem]{Example}
\newtheorem{definition}[theorem]{Definition}

\newtheorem{remark}[theorem]{Remark}

\begin{document}

\author[Peter Danchev]{Peter Danchev}
\address{Institute of Mathematics and Informatics, Bulgarian Academy of Sciences, 1113 Sofia, Bulgaria}
\email{danchev@math.bas.bg; pvdanchev@yahoo.com}

\author[A. Javan]{Arash Javan}
\address{Department of Mathematics, Tarbiat Modares University, 14115-111 Tehran Jalal AleAhmad Nasr, Iran}
\email{a.darajavan@modares.ac.ir; a.darajavan@gmail.com}

\author[O. Hasanzadeh]{Omid Hasanzadeh}
\address{Department of Mathematics, Tarbiat Modares University, 14115-111 Tehran Jalal AleAhmad Nasr, Iran}
\email{hasanzadeomiid@gmail.com}

\author[A. Moussavi]{Ahmad Moussavi}
\address{Department of Mathematics, Tarbiat Modares University, 14115-111 Tehran Jalal AleAhmad Nasr, Iran}
\email{moussavi.a@modares.ac.ir; moussavi.a@gmail.com}

\title[Generalized Nil-Clean Rings]{Rings Whose Non-Invertible Elements are Nil-Clean}

\maketitle




\begin{abstract}
We systematically study those rings whose non-units are a sum of an idempotent and a nilpotent. Some crucial characteristic properties are completely described as well as some structural results for this class of rings are obtained. This work somewhat continues two publications on the subject due to Diesl (J. Algebra, 2013) and Karimi-Mansoub et al. (Contemp. Math., 2018).
\end{abstract}

\section{Introduction and Motivation}

Everywhere in the current paper, let $R$ be an associative but {\it not} necessarily commutative ring having identity element, usually denoted as $1$. Standardly, for such a ring $R$, the letters $U(R)$, $Nil(R)$ and $Id(R)$ are designed for the set of invertible elements (also termed as the unit group of $R$), the set of nilpotent elements and the set of idempotent elements in $R$, respectively. Likewise, $J(R)$ denotes the Jacobson radical of $R$, and $Z(R)$ denotes the center of $R$. The ring of $n \times n$ matrices over $R$ and the ring of $n \times n$ upper triangular matrices over $R$ are stand by $M_n(R)$ and $T_n(R)$, respectively. Traditionally, a ring is said to be {\it abelian} if each of its idempotents is central, that is, $Id(R) \subseteq Z(R)$.

Imitating \cite{nic}, an element $r$ in a ring $R$ is said to be {\it clean} if there is an idempotent $e \in R$ such that $r-e\in U(R)$, and a clean ring is defined as the ring in which every element is clean. On the other hand, mimicking \cite{diesl}, an element $r$ in a ring $R$ is said to be {\it nil-clean} if there is an idempotent $e \in R$ such that $r-e\in Nil(R)$, and a nil-clean ring is defined as the ring in which each element is nil-clean. In the case where $re=er$, we call these rings {\it strongly (nil-)clean}.

It is well know that there exists a nil-clean element that is {\it not} clean, but however all nil-clean rings are necessarily clean. The study of (strongly) (nil-)clean elements and rings has gained a significant attention in the past decade as evidenced by the existing articles \cite{BCDM,chenshib,DHM,DHJM,diesl,hannic,kosan1,sharp}.

In the same vein, continuing the cited above references, in \cite{K-M} were investigated those rings whose invertible elements are a sum of an idempotent and a nilpotent, i.e., $U(R)\subseteq Id(R)+Nil(R)$. This is a common generalization of the so-called {\it UU rings}, explored in detail in \cite{CUU} and \cite{DL}, that are rings with $U(R)=1+Nil(R)$.

Our aim, which motivates writing of this paper, is to examine what will happen in the dual case when non-units in rings are nil-clean elements, thus somewhat also expanding nil-clean rings in an other way. Recently, we studied 
in-depth the same situation, but related for strongly nil-clean rings (see \cite{DHJM}). 

So, we now arrive at our key instrument introduced as follows.

\begin{definition} We call a ring $R$ a {\it generalized nil-clean}, briefly abbreviated by {\it GNC}, provided $$R\backslash U(R)\subseteq Id(R) + Nil(R).$$
\end{definition}

Our further work is organized in the following two directions: In the next section, we achieve to exhibit some major properties and characterizations of GNC rings in various different aspects (see, for instance, Theorems~\ref{2-primal}, \ref{field}, \ref{semilocal} and \ref{morita}). In the subsequent section, we explore when a group ring is GNC under some restrictions on the former group and ring (see, e.g., Lemma~\ref{rg} and Theorem~\ref{gr}).

\section{Examples and Basic Properties of GNC Rings}

We begin here with some trivial, but useful assertions.

\begin{lemma}
Let $R$ be a ring and $a\in R$ be a nil-clean element. Then, $-a$ is clean.
\end{lemma}

\begin{proof}
Assume $a = e + q$ is a nil-clean representation. Thus, $-a = (1-e) - (1+q)$ is a clean representation.
\end{proof}

\begin{corollary}\label{clean}
Let $R$ be a GNC ring. Then, $R$ is clean.
\end{corollary}

In other words, GNC rings lie between the nil-clean rings and clean rings. However, it is worthwhile noticing that $\mathbb{Z}_3$ is a GNC ring that is manifestly {\it not} nil-clean, while $\mathbb{Z}_6$ is a clean ring that is {\it not} GNC.

\begin{lemma}\label{J(R) nil}
Let $R$ be a GNC ring. Then, $J(R)$ is nil.
\end{lemma}

\begin{proof}
Choose $j \in J(R)$. Since $j \notin U(R)$, we have $e = e^2 \in R$ and $q \in Nil(R)$ such that $j = e + q$. Therefore, $1-e = (q+1) - j \in U(R) + J(R) \subseteq U(R)$, so $e = 0$. Hence, $j = q \in Nil(R)$, as required.
\end{proof}

The next constructions are worthy of documentation.

\begin{example}
For any ring $R$, the polynomial ring $R[x]$, the Laurent polynomial ring $R[x,x^{-1}]$, and the formal power series ring $R[[x]]$ are all {\it not} GNC rings.
\end{example}

\begin{proof}
Assuming that $R[x,x^{-1}]$ is a GNC ring, then $1+x$ is not a unit. This is because, if $1+x$ is a unit, there exists a suitable power of x such that
$$(1+x)(a_0+a_1x+ \cdots + a_nx^n)=x^k,$$
for some $n > 1$ and $k$, where $a_0$ and $a_n$ are non-zero. The left-hand side contains the two distinct terms  $a_0$ and $a_n$ , contradicting the equality. Therefore, there exists $e=e^2 \in R[x,x^{-1}]$ and $q \in Nil(R[x,x^{-1}])$ such that $1+x=e+q$. Consequently, $1-e = x(1+x^{-1}q) \in U(R[x,x^{-1}])$, which means $1+x = q \in Nil(R[x,x^{-1}])$, that is the desired contradiction.

Supposing now that $R[[x]]$ is GNC ring, we know that $$J(R[[x]]) = \{a + xf(x) : a \in J(R) \text{ and } f(x) \in R[[x]]\}$$ (see, for example, \cite[Exercise 5.6]{lam fi}), and so it is clear that $x \in J(R[[x]])$. Therefore, $J(R[[x]])$ is not nil, a contradiction, thus getting the desired claim.

Next, if $R[x]$ is GNC, then it is clean in view of Lemma \ref{clean}. But then this contradicts \cite[Remark 2.8]{42}, and hence $R[x]$ cannot be GNC, as claimed.
\end{proof}

In regard to the above example, in the following theorem we attempt to classify when a ring is $2$-primal in terms of nil-clean elements, respectively in rings $R[x]$, $R[x, x^{-1}]$ and $R[[x]]$. Concretely, the following curious statement is true:

\begin{theorem}\label{2-primal}
Let $R$ be a ring and $NC(R)$ be the set of all nil-clean elements in the ring $R$. Then, we have:

(1) $R$ is a $2$-primal ring if, and only if, $NC(R[x]) = NC(R) + Nil_*(R)[x]x$.

(2) $R$ is a $2$-primal ring if, and only if, $NC(R[x, x^{-1}]) = Nil_{\ast}(R)[x, x^{-1}]x^{-1} + NC(R) + Nil_{\ast}(R)[x, x^{-1}]x $.

(3) $R$ is a $2$-primal ring if, and only if, $NC(R[[x]]) \subseteq NC(R) + Nil_*(R)[[x]]x$.
\end{theorem}

\begin{proof}
We need only to prove points (1) and (3), as the proof of (2) is similar to that of (1).

(1) Assume $R$ is a $2$-primal ring. If $f = \sum_{i=0}^{n} a_ix^i \in NC(R[x])$, then there exist $e = \sum_{i=0}^{n} e_ix^i \in Id(R[x])$ and $q = \sum_{i=0}^{n} q_ix^i \in Nil(R[x])$ such that $f = e + q$. Clearly, $e_0 \in Id(R)$ and $q_0 \in Nil(R)$. Since $R$ is $2$-primal, the quotient $R/Nil_*(R)$ is reduced. Therefore, according to \cite[Theorem 5]{kanwar}, we deduce $\overline{e} = \sum_{i=0}^{n} \overline{e}_ix^i = \overline{e}_0$, which implies that, for every $i \geq 1$, $e_i \in Nil_*(R)$. Moreover, since $R$ is $2$-primal, it follows that $R[x]$ is $2$-primal. Thus, one sees that $$q \in Nil(R[x]) = Nil_*(R[x]) = Nil_*(R)[x]$$ whence, for every $i \geq 0$, $q_i \in Nil_*(R)$. Since, for each $i \geq 0$, $a_i = e_i + q_i$, it follows that, for each $i \geq 1$, $a_i \in Nil_*(R)$. Now, if $f = \sum a_ix^i \in NC(R) + Nil_*(R)[x]x$, then we can write $a_0 = e + q$, where $e \in Id(R)$ and $q \in Nil(R)$. Therefore, $$f = e + (q + a_1x + \cdots + a_nx^n).$$ It, thereby, suffices to show that $$q + a_1x + \cdots + a_nx^n \in Nil(R[x]),$$ which is obvious, because $$q + a_1x + \cdots + a_nx^n \in Nil(R)[x] = Nil_*(R)[x] = Nil_*(R[x]) \subseteq Nil(R[x]).$$

Reciprocally, suppose $$NC(R[x]) = NC(R) + Nil_*(R)[x]x.$$ If $a \in Nil(R)$, then $ax \in Nil(R[x])$. Consequently, $ax$ is a nil-clean element in $R[x]$. By assuming $$ax = b_0 + b_1x + \cdots + b_nx^n \in NC(R) + Nil_*(R)[x]x,$$ we find that $a = b_1 \in Nil_*(R)$, which guarantees that $R$ is $2$-primal, as wanted.

\medskip

(3) {\bf Claim:} If $R$ is a $2$-primal ring, then $Nil(R[[x]]) \subseteq Nil_*(R)[[x]]$.

\medskip

Since $R$ is a $2$-primal ring, one verifies that $R/Nil_*(R)$ is reduced. Thus, $R/Nil_*(R)[[x]]$ must also be reduced. Moreover, since $$R[[x]]/Nil_*(R)[[x]] \cong R/Nil_*(R)[[x]],$$ it follows at once that $R[[x]]/Nil_*(R)[[x]]$ is reduced. Hence, $Nil(R[[x]]) \subseteq Nil_*(R)[[x]]$, as claimed.

The rest of the proof follows by arguing as in (1).
\end{proof}

Before proceed by proving the main theorem listed below, we need to establish two helpful assertions.

\begin{proposition}
Let $R$ be a GNC ring. Then, for every $n \in \mathbb{N}$, either $n \in Nil(R)$ or $n \in U(R)$.
\end{proposition}

\begin{proof}
By using induction on $n$, we shall prove the statement for $n$. In fact, if $n=1$, the proof is evident. Assume, for $k < n$, we have either $k \in Nil(R)$ or $k \in U(R)$. If $k+1 \in U(R)$, there is nothing left to prove. So, assume $k+1 \notin U(R)$. Thus, there exist $e \in Id(R)$ and $q \in Nil(R)$ such that $k+1 = e + q$. Since $k+1 \notin U(R)$, we must have $k \notin Nil(R)$, so by the induction hypothesis it must be that $k \in U(R)$. Therefore, we infer $1-e = -k(1-k^{-1}q) \in U(R)$, so $e=0$, which forces $k+1 = q \in Nil(R)$, as expected.
\end{proof}

\begin{lemma}\label{commut}
Let $R$ be a GNC ring with $2 \in U(R)$ and, for every $u \in U(R)$, we have $u^2 = 1$. Then, $R$ is a commutative ring.
\end{lemma}

\begin{proof}
For any $u, v \in U(R)$, we have $u^2 = v^2 = (uv)^2 = 1$. Therefore, $uv = (uv)^{-1} = v^{-1}u^{-1} = vu$. Hence, the units commute with each other.

Now, we will show that $R$ is abelian. Indeed, for every $e \in \text{Id}(R)$ and $r \in R$, we know $2e-1\in U(R)$ and $(1+er(1-e)) \in U(R)$. Since the units commute with each other, we have $2er(1-e) = 2(1-e)re = 0$. Since $2 \in U(R)$, we derive $er(1-e) = (1-e)re = 0$, which insures $er = ere = re$. Therefore, $R$ is abelian, as pursued.

On the other side, since $1 + \text{Nil}(R) \subseteq U(R)$ and the units commute with each other as showed above, the nilpotent elements also commute with each other.

Furthermore, we will demonstrate that $R$ is commutative. To this goal, let $x, y \in R$. We distinguish the following four cases:
\begin{enumerate}
\item $x, y \in U(R)$: since the units commute with each other, it is clear that $xy = yx$.
\item $x, y \notin U(R)$: since $R$ is a GNC ring, there exist $e, f \in \text{Id}(R)$ and $p, q \in \text{Nil}(R)$ such that $x = e + q$ and $y = f + p$. Thus, we extract
    \[
    xy = (e + q)(f + p) = ef + ep + qf + qp = fe + pe + fq + pq = (f + p)(e + q) = yx.
    \]
\item $x \in U(R)$ and $y \notin U(R)$: in this case, there exists $e \in \text{Id}(R)$ and $q \in \text{Nil}(R)$ such that $y = e + q$. Since $x(1 + q) = (1 + q)x$, we have $qx = xq$. Therefore, we obtain
    \[
    xy = x(e + q) = xe + xq = ex + qx = (e + q)x = yx.
    \]
\item $x \notin U(R)$ and $y \in U(R)$: analogously to case (3), we can show that $xy = yx$.
\end{enumerate}

Consequently, bearing in mind all of the presented above, $R$ is a commutative ring, as asked for.
\end{proof}

\begin{remark} In the lemma above, the condition $2\in U(R)$ is necessary and cannot be omitted. This is because, assuming $R=T_2(\mathbb{Z}_2)$, we have that $R$ is a GNC ring and, for every $u \in R$, $u^2=1$, but $R$ is obviously {\it not} commutative. If, however, we do {\it not} consider the condition $2 \in U(R)$, equipped with Lemma \ref{NR} situated below, we can show that either $R$ is a commutative local ring or $R$ is a strongly nil-clean ring. But, if we examine the assumptions of Lemma \ref{commut} more exactly, we can readily show that $R$ is a field, as stated in the proof of the next principal theorem.
\end{remark}

We now have at hand all the machinery necessary to show truthfulness of the following statement which considerably extends Lemma~\ref{commut} from commutative rings to fields. We, however, will give a more conceptual and transparent proof without its direct usage.

\begin{theorem}\label{field}
Let $R$ be a GNC ring with $2 \in U(R)$ and, for each $u \in U(R)$, we have $u^2 = 1$. Then, $R$ is a field.
\end{theorem}

\begin{proof}
Firstly, we show that the ring $R$ does {\it not} have non-trivial idempotent and nilpotent elements.

To this purpose, assume $q \in \text{Nil}(R)$. Then, $(1 \pm q) \in U(R)$, so that $$1-2q +q^2=(1-q)^2=1=(1+q)^2=1+2q+q^2.$$ Therefore, $4q=0$. Since $2 \in U(R)$, we conclude that $q=0$.

Now, we show that $R$ does {\it not} have non-trivial idempotents. In fact, if $0,1 \neq e \in \text{Id}(R)$, then $2e \notin U(R)$. If, for a moment, $2e \in U(R)$, since $2 \in U(R)$, we get $e\in Id(R) \cap U(R)={1}$, leading to a contradiction. Thus, $2e \notin U(R)$. But since $R$ is a GNC ring and as observed above $\text{Nil}(R)=\{0\}$, we detect $2e=f \in \text{Id}(R)$. So, $4e=2e$, which assures $2e=0$. Again, since $2 \in U(R)$, we can get $e=0$, which is a contradiction. Hence, $R$ does {\it not} have non-trivial idempotents, as asserted.

Now, we illustrate that $R$ is a division ring. To this target, since by what we have seen above $R$ does not have non-trivial nilpotent and idempotent elements, for any $x \notin U(R)$, we must have either $x=0$ or $x=1$. Since $x \notin U(R)$, it must be that $x=0$. Therefore, $R$ is a division ring, as claimed.

Furthermore, since for any $u \in U(R)$ we have $u^2=1$, and $R$ is a division ring, this enables us that $a^3=a$ for every $a \in R$. Hence, the classical Jacobson's theorem allows us to conclude that $R$ is a commutative ring, as promised.
\end{proof}

As a consequence, we yield:

\begin{corollary}\label{sum nil}
Let $R$ be a GNC ring. Then, $Nil(R)+J(R)=Nil(R)$.
\end{corollary}

\begin{proof}
Let us assume that $a \in Nil(R)$ and $b \in J(R)$. Thus, there exists an $n \in \mathbb{N}$ such that $a^n = 0$. Therefore, $(a+b)^n = a^n + j$, where $j \in J(R)$. Employing Lemma \ref{J(R) nil}, we arrive at $(a+b)^n = j \in Nil(R)$, as required.
\end{proof}

Our next reduction criterion is this one:

\begin{proposition}\label{factor}
(1) For any nil-ideal $I \subseteq R$, $R$ is GNC if, and only if, $R/I$ is GNC.

(2) A ring $R$ is GNC if, and only if, $J(R)$ is nil and $R/J(R)$ is GNC.

(3) The direct product $\prod_{i=1}^{n} R_i$ is GNC for $n \ge 2$ if, and only if, each $R_i$ is nil-clean.
\end{proposition}

\begin{proof}
(1) We assume that $\overline{R} = R/I$ and $\bar{a} \notin U(\overline{R})$. Then, $a \notin U(R)$, so there exist $e \in Id(R)$ and $q \in Nil(R)$ such that $a = e + q$. Thus, $\bar{a} = \bar{e} + \bar{q}$.

Conversely, let us assume that $\overline{R}$ is a GNC ring. We also assume that $a \notin U(R)$, so $\bar{a} \notin U(\overline{R})$, hence $\bar{a} = \bar{e} + \bar{q}$, where $\bar{e} \in Id(\overline{R})$ and $\bar{q} \in Nil(\overline{R})$. Since $I$ is a nil-ideal, we can assume that $e \in Id(R)$ and $q \in Nil(R)$. Therefore, $a - (e + q) \in I \subseteq J(R)$, so that there exists $j \in J(R)$ such that $a = e + (q + j)$. Hence, in virtue of Corollary \ref{sum nil}, $a$ has a nil-clean representation, as needed.

(2) Utilizing Lemma \ref{J(R) nil} and part (1), the conclusion is apparent.

(3) Letting each $R_i$ be nil-clean, then $\prod_{i=1}^{n} R_i$ is nil-clean applying \cite[Proposition 3.13]{diesl}. Hence, $\prod_{i=1}^{n} R_i$ is GNC.

Oppositely, assume that $\prod_{i=1}^{n} R_i$ is GNC and that $R_j$ is not nil-clean for some index $1 \le j \le n$. Then, there exists $a \in R_j$ which is not nil-clean. Consequently, $(0, \cdots  , 0, a, 0, \cdots , 0)$ is not a nil-clean element in $\prod_{i=1}^{n} R_i$. But, it is plainly seen that $(0, \cdots  , 0, a, 0, \cdots , 0)$ is not invertible in $\prod_{i=1}^{n} R_i$ and thus, by hypothesis, $(0, \cdots  , 0, a, 0, \cdots , 0)$ is in turn not nil-clean, a contradiction. Therefore, each $R_i$ is nil-clean, as formulated.
\end{proof}

As a consequence, we derive:

\begin{corollary}\label{image}
Every homomorphic image of a GNC ring is again GNC.
\end{corollary}

\begin{proof}
Let us assume that \( I \) is an ideal of \( R \). We consider \( \overline{R} = R/I \). Clearly, for every \( a + I \notin U(\overline{R}) \), we have \( a \notin U(R) \), as expected.
\end{proof}

Assuming that \(L_n(R) = \left\{
\begin{pmatrix}
    0 & \cdots & 0 & a_1 \\ 0 & \cdots & 0 & a_2 \\ \vdots & \ddots & \vdots & \vdots \\ 0 & \cdots & 0 & a_n
\end{pmatrix}
\in T_n(R) : a_i \in R \right \} \subseteq T_n(R)\) and \(S_n(R) = \{(a_{ij}) \in T_n(R) : a_{11}= \cdots = a_{nn}\} \subseteq T_n(R)\), it is not so hard to check that the mapping $\varphi: S_n(R) \to S_{n-1}(R) \propto  L_{n-1}(R)$, defined as
$$\varphi \left(
\begin{pmatrix}
    a_{11} & a_{12} & \cdots & a_{1n}\\
    0 & a_{11} & \cdots & a_{2n}\\
    \vdots  & \vdots & \ddots  & \vdots \\
    0 & 0 & \cdots & a_{11}
\end{pmatrix}
\right) =
\begin{pmatrix}
    a_{11} & a_{12} & \cdots & a_{1,n-1} & 0 & \cdots & 0 & a_{1n} \\
    0 & a_{11} & \cdots & a_{2,n-1} & 0 & \cdots & 0 & a_{2n} \\
    \vdots  & \vdots  & \ddots  & \vdots  & \vdots  & \ddots  & \vdots  & \vdots  \\
    0 & 0 & \cdots & a_{11} & 0 & \cdots & 0 & a_{n-1,n} \\
    0 & 0 & \cdots & 0 & a_{11} & a_{12} & \cdots & a_{1,n-1} \\
    0 & 0 & \cdots & 0 & 0 & a_{11} & \cdots & a_{2,n-1} \\
    \vdots & \vdots & \ddots  & \vdots & \vdots  & \vdots  & \ddots  & \vdots \\
    0 & 0 & \cdots & 0 & 0 & 0 & \cdots & a_{11} \\
\end{pmatrix},$$
means $S_n(R) \cong S_{n-1}(R) \propto  L_{n-1}(R)$. Note that this isomorphism provides a suitable tool to study the ring $S_n(R)$, because by examining the trivial extension and using induction on $n$, we can extend the result to $S_n(R)$. Specifically, we are able to establish validity of the following.

\begin{corollary}
Let $R$ be ring, and $M$ a bi-module over $R$. Then the following items hold:

(1) The trivial extension $R \propto  M$ is a GNC ring if, and only if, $R$ is a GNC ring.

(2) For $n \ge 2$, $S_n(R)$ is GNC ring if, and only if, $R$ is a GNC.

(3) For $n \ge 2$, $R_n:=R[x]/(x^n)$ is GNC ring if, and only if, $R$ is a GNC.

(4) For $n,m \ge 2$, $A_{n,m}(R):=R[x,y \mid x^n=yx=y^m=0]$ is GNC ring if, and only if, $R$ is a GNC.

(5) For $n,m \ge 2$, $B_{n,m}(R):=R[x,y \mid x^n=y^m=0]$ is GNC ring if, and only if, $R$ is a GNC.

\end{corollary}

\begin{proof}
(1) We take $I=0 \propto M$, so clearly $I$ is a nilpotent ideal of $R\propto M$ and $\frac{R\propto M}{I} \cong R$, hence the proof is complete exploiting Theorem \ref{factor}(1).

(2) {\bf Method 1:} We assume $I=\{(a_{ij}) \in S_n(R) : a_{11}=0\}$, so evidently $I$ is a nilpotent ideal of $S_n(R)$ and $S_n(R)/I \cong R$.

\noindent{\bf Method 2:} We shall prove the problem by induction on $n$. Assuming $n=2$, then $S_2(R)=R\propto R$, so the proof is straightforward thanks to (1). Now, assuming the problem holds for $k<n$, and since $S_{k+1}(R) \cong S_k(R)\propto L_k(R)$, we again deduce that $S_{k+1}(R)$ is a GNC ring in view of (1).

(3) We assume $$I=\{\sum_{i=0}^{n-1} a_{i}x^i \in R_{n} : a_0=0\},$$ so obviously $I$ is a nilpotent ideal of $R_{n}$ and we infer $R_{n}/I \cong R$.

(4) We assume $$I=\{a+\sum_{i=1}^{n-1} b_{i}x^i+ \sum_{j=1}^{m-1}c_{j}y^j \in A_{n,m}(R) : a=0\},$$ so immediately $I$ is a nilpotent ideal of $A_{n,m}(R)$ and we conclude $A_{n,m}(R)/I \cong R$.

(5) We assume $$I=\{\sum_{i=0}^{n-1}\sum_{j=0}^{m-1} a_{ij}x^iy^j \in B_{n,m}(R) : a_{00}=0\},$$ so automatically $I$ is a nilpotent ideal of $B_{n,m}(R)$ and we derive $B_{n,m}(R)/I \cong R$.
\end{proof}

We now continue our work with the following necessary and sufficient condition.

\begin{proposition}\label{local}
Let $R$ be a ring with only trivial idempotents. Then, $R$ is GNC if, and only if, $R$ is a local ring with $J(R)$ nil.
\end{proposition}

\begin{proof}
Assuming $R$ is a GNC ring, Lemma \ref{J(R) nil} ensures that $J(R)$ is nil. Now, if $a \notin U(R)$, then we have either $a = 1 + q$ or $a = 0 + q$, where $q \in Nil(R)$. Since $a$ is not a unit, it must be that $a = 0 + q$, implying $a = q \in Nil(R)$. Thus, by virtue of \cite[Proposition 19.3]{lam fi}, $R$ is a local ring.

Now, conversely, suppose $R$ is a local ring with a nil Jacobson radical $J(R)$. So, for each $a \notin U(R)$, we have $a \in J(R) \subseteq Nil(R)$, whence $a$ is a nil-clean element.
\end{proof}

Before establishing our next main result quoted below concerning semi-local rings, two more technicalities are in order.

\begin{lemma}\cite[Theorem 3]{kosan} \label{div}
Let $D$ be a division ring. If $|D| \ge 3$ and $a \in D \backslash \{0, 1\}$, then
$\begin{pmatrix}
    a & 0 \\0 & 0
\end{pmatrix} \in M_n(D)$
is not nil-clean.
\end{lemma}

\begin{lemma} \label{div GNC}
Let $n\ge 2$ and let $D$ be a division ring. Then, the matrix ring $M_n(D)$ is a GNC ring if, and only if, $D \cong \mathbb{Z}_2$.
\end{lemma}

\begin{proof}
If foremost $D \cong \mathbb{Z}_2$, then with the aid of \cite{BCDM} the ring $M_n(D)$ is GNC.

Next, conversely, if $M_n(D)$ is a GNC ring and $|D| \ge 2$, then Lemma \ref{div} gives that, for every $a \in D \backslash \{0,1\}$, the element $\begin{pmatrix}
    a & 0 \\0 & 0
\end{pmatrix} \in M_n(D)$ is simultaneously not a unit in $M_n(D)$ and not nil-clean, leading to a contradiction. Therefore, it must be that $D \cong \mathbb{Z}_2$, as given in the text.
\end{proof}

Notice that Diesl raised the question in \cite{diesl} whether a matrix ring over a nil-clean ring is also nil-clean. However, based on the last lemma, this question definitely has a negative answer for the class of GNC rings, as $\mathbb{Z}_3$ is a GNC ring, but, for any $n\ge 2$, $M_n(\mathbb{Z}_3)$ is {\it not} a GNC ring.

\medskip

We are now ready to prove the mentioned above result.

\begin{theorem}\label{semilocal}
Let $R$ be a ring. Then, the following conditions are equivalent for a semi-local ring:

(1) $R$ is a GNC ring.

(2) Either $R$ is a local ring with a nil Jacobson radical, or $R/J(R) \cong M_n(\mathbb{Z}_2)$ with a nil Jacobson radical, or $R$ is a nil-clean ring.
\end{theorem}

\begin{proof}
(2) $\Rightarrow$ (1). The proof is straightforward by combination of Proposition \ref{local} and Lemma \ref{div GNC}.

(1) $\Rightarrow$ (2). Since $R$ is semi-local, we write $R/J(R) \cong \prod_{i=1}^{m}M_{n_i}(D_i)$, where each $D_i$ is a division ring. Moreover, the application of Theorem \ref{factor}(2) leads to $J(R)$ is nil, and $R/J(R)$ is a GNC ring. If $m = 1$, then by Lemma \ref{div GNC} we have either $R/J(R) = D_1$ or $R/J(R) \cong M_n(\mathbb{Z}_2)$. If $m > 1$, then by Theorem \ref{factor}(3), for each $1 \le i \le m$, the ring $M_{n_i}(D_i)$ is nil-clean. Finally, referring to Lemma \ref{div GNC}, for any $1 \le i \le m$, we have $D_i \cong \mathbb{Z}_2$. Thus, \cite[Corollary 6]{kosan} applies to conclude that $R$ is a nil-clean ring.
\end{proof}

As an immediate consequence, we find:

\begin{corollary} \label{artinian}
Let $R$ be an Artinian (in particular, a finite) ring. Then, the following conditions are equivalent:

(1) $R$ is a GNC ring.

(2) Either $R$ is a local ring with a nil Jacobson radical, or $R$ is a nil-clean ring.
\end{corollary}

\begin{example}
In the special case of the ring $\mathbb{Z}_n$, where $n\in \mathbb{N}$, it is GNC if, and only if, $n$ is power of a prime number.
\end{example}

Our next two consequences are these:

\begin{corollary} \label{semi-simple}
Let $R$ be a ring. Then, the following issues are equivalent for a semi-simple ring $R$:

(1) $R$ is a GNC ring.

(2) $R$ is either a division ring, or a nil-clean ring.
\end{corollary}

\begin{corollary}
Suppose \( n \ge 2 \) and \( R \) is a semi-local (or an Artinian or a semi-simple) ring. Then, the ring \( M_n(R) \) is a GNC ring if, and only if, \( M_n(R) \) is a nil-clean ring.
\end{corollary}

\begin{proof}
Taking into account Theorem \ref{semilocal}, Corollary \ref{artinian} and Corollary \ref{semi-simple}, nothing remains to be proven.
\end{proof}

The next claim is well-known, but we state it here only for the sake of completeness and the readers' convenience. Following \cite{nic}, recall also that a ring $R$ is {\it exchange}, provided that, for any $a$ in $R$, there is an idempotent $e \in R$ such that $e \in aR$ and $1 - e \in (1 - a)R$.

\begin{lemma}\label{domain}
Let $R$ be an exchange domain with $J(R) = \{0\}$, then $R$ is a division ring.
\end{lemma}

\begin{proof}
Assume that \( a \in R \backslash U(R) \) and \( x \) is an arbitrary element of \( R \). Then, there exists \( e \in Id(R) \) such that \( e \in Rax \) and \( 1-e \in R(1-ax) \). Since \( R \) is a domain, we have either \( e=0 \) or \( e=1 \). Because \( a \notin U(R) \), it follows that \( e=0 \). Therefore, \( a \in J(R)=0 \). Thus, \( R \) must be a division ring, as stated.
\end{proof}

We are now in a position to provide a confirmation of the following statement.

\begin{proposition} \label{matrix and 2-primal}
Let $R$ be a $2$-primal ring and $n \ge 2$. Then, $M_n(R)$ is GNC if, and only if, $R/J(R)$ is Boolean and $J(R)$ is nil.
\end{proposition}

\begin{proof}
$(\Leftarrow).$ The proof is straightforward using \cite[Theorem 6.1]{kosan1}.

$(\Rightarrow).$ Assume $M_n(R)$ is a GNC ring. Since $M_n(J(R)) = J(M_n(R))$ (see cf. \cite{lam fi}), Lemma \ref{J(R) nil} is a guarantor that $J(R)$ is nil. Now, we menage to show that $R/J(R)$ is a Boolean ring. Indeed, since $R$ is a $2$-primal ring, we have $Nil_{*}(R) = J(R) = Nil(R)$, so that the factor-ring $R/J(R)$ is reduced. Therefore, $R/J(R)$  is a sub-direct product of a family of domains $\{S_i\}_{i \in I}$. As being a homomorphic image of $M_n(R/J(R))$, the ring $M_n(S_i)$ is also GNC in accordance with Corollary~\ref{image}, and hence it is clean with the help of Lemma \ref{clean}. Thus, $M_n(S_i)$ is an exchange ring, so by \cite[Proposition 1.10 ]{nic}, for each $i \in I$, $S_i$ is an exchange ring too. So, adapting Lemma \ref{domain}, for each $i \in I$, $S_i$ must be a division ring. Likewise, since $M_n(S_i)$ is a GNC ring, knowing Lemma \ref{div GNC}, for each $i \in I$, it must be that $S_i \cong \mathbb{Z}_2$. This, after all, implies that $R/J(R)$ is a Boolean ring, as promised.
\end{proof}

Since it is well-established in \cite{DL} and \cite{kosan1} that a ring $R$ is strongly nil-clean precisely when the quotient-ring $R/J(R)$ is Boolean and the ideal $J(R)$ is nil, the next consequence is directly fulfilled. However, we now intend to give an independent verification as follows.

\begin{corollary}
Let $R$ be a $2$-primal ring and $n \ge 2$. Then, $M_n(R)$ is GNC if, and only if, $R$ is a (strongly) nil-clean ring.
\end{corollary}

\begin{proof}
With Proposition \ref{matrix and 2-primal} at hand, it is sufficient to show that if $R$ is a nil-clean ring, then $M_n(R)$ is a GNC ring. Since $R$ is nil-clean, $J(R)$ is nil by \cite[Proposition 3.16]{diesl}. Now, by Proposition \ref{matrix and 2-primal}, it is enough to show that $R/J(R)$ is Boolean. Since $R$ is a 2-primal ring, $R/J(R)$ is abelian. On the other hand, any abelian nil-clean ring is strongly nil-clean, so $R/J(R)$ is strongly nil-clean. Therefore, by \cite[Theorem 2.5]{chenshm}, $R/J(R)$ is Boolean.
\end{proof}

\begin{remark} It was proved in \cite[Corollary 2.12]{Diran} that if $R$ is a $2$-primal strongly nil-clean ring, then $M_n(R)$ is nil-clean (and thus GNC) for each $n \geq 1$.
\end{remark}

We now have all the information needed to prove the following.

\begin{lemma} \label{abelian}
Let $R$ be an abelian ring. Then, the following conditions are equivalent:

(1) $R$ is a GNC ring.

(2) Either $R$ is a local ring with nil Jacobson radical, or $R$ is a strongly nil-clean ring.
\end{lemma}

\begin{proof}
The implication (2) $\Rightarrow$ (1) is simple, so we remove its inspection.

(1) $\Rightarrow$  (2). Let $R$ be a GNC ring that is not local. Then, with the aid of Proposition \ref{local}, there exists a non-trivial idempotent $e \in Id(R)$. Since $R$ is abelian, we have $R=eRe\oplus(1-e)R(1-e)$. Therefore, Theorem \ref{factor}(3) is applicable to deduce that both the rings $eRe$ and $(1-e)R(1-e)$ are nil-clean. Thus, using \cite[Proposition 3.13]{diesl}, $R$ is a nil-clean ring. But, as $R$ is abelian, it must be strongly nil-clean.
\end{proof}

As usual, we call a ring an {\it NR} (resp., an {\it NI}) ring if its set of nilpotent elements forms a subring (resp., an ideal).

\begin{lemma} \label{NR}
Let $R$ be an NR ring. Then, the following two conditions are equivalent:

(1) $R$ is a GNC ring.

(2) $R$ is either local with nil $J(R)$, or $R$ is strongly nil-clean.
\end{lemma}

\begin{proof}
The implication (2) $\Rightarrow$ (1) is elementary.

(1) $\Rightarrow$ (2). Assume $R$ is a GNC ring. Thus, by Lemma \ref{clean}, $R$ is an exchange ring. Therefore, thanks to \cite[Corollary 2.17]{chenlin}, $R/J(R)$ is an abelian ring. Also, Theorem \ref{factor}(2) employs $R/J(R)$ is a GNC ring. Hence, Lemma \ref{abelian} helps to deduce that either $R/J(R)$ is local or strongly nil-clean ring. Finally, \cite[Theorem 2.5]{chenshm} is in hand to conclude that either $R$ is local or strongly nil-clean, as required.
\end{proof}

The following consequence is somewhat a little surprising.

\begin{corollary}
Let $R$ be a GNC ring. Then, $R$ is an NR ring if, and only if, $R$ is an NI ring. In particular, if $R$ is a GNC ring and is an NR ring, then $J(R) = Nil(R)$.
\end{corollary}

We now begin by establishing some structural results presented in the sequel.

\begin{lemma}
Let $R$ be a ring. Then, $R$ is strongly nil-clean if, and only if, $R$ is both UU and GNC.
\end{lemma}

\begin{proof}
If $R$ is strongly nil-clean, then one elementarily sees that $R$ is both a UU-ring and a GNC-ring. Therefore, it is sufficient to prove the converse.

To do that, let us assume that $R$ is simultaneously UU and GNC. According to \cite[Corollary 2.13]{ster}, $R$ is an NR ring. Thus, owing to Lemma \ref{NR}, $R$ is either local with nil $J(R)$ or $R$ is a strongly nil-clean ring.

Now, let us suppose that $R$ is a local with nil $J(R)$. If $a \in U(R)$, then $a$ is strongly nil-clean, because $R$ is an UU-ring. On the other side, if $a \notin U(R)$, we have $a \in J(R) \subseteq Nil(R)$, giving that $a$ is strongly nil-clean, as needed.
\end{proof}

\begin{proposition} \label{matrix and NR}
Let $R$ be an NR ring and $n \ge 2$. Then, $M_n(R)$ is GNC if, and only if, $R/J(R)$ is Boolean and $J(M_n(R))$ is nil.
\end{proposition}

\begin{proof}
$(\Leftarrow).$  Since \( R/J(R) \) is a Boolean ring, consulting with \cite[Corollary 6]{BCDM}, the factor
\( M_n(R)/J(M_n(R)) \cong M_n(R/J(R)) \) is a nil-clean ring. And since \( J(M_n(R)) \) is nil, Theorem \ref{factor}(1) yields that \( M_n(R) \) is a GNC ring.

$(\Rightarrow).$ Since \( M_n(R) \) is a GNC ring, by Lemma \ref{J(R) nil}, \( J(M_n(R))=M_n(J(R)) \) is nil as well. Also, \( M_n(R/J(R)) \) is GNC. Therefore, \( R/J(R) \) is an exchange ring. And since \( J(R) \) is nil, we conclude that \( R/J(R) \) is also an NR ring. Consequently, handling \cite[Proposition 2.19]{chenlin}, \( R/J(R) \) is reduced and hence $2$-primal. Thus, Proposition \ref{matrix and 2-primal} guarantees that \( R/J(R) \) is Boolean, ending the conclusion.
\end{proof}

\begin{proposition} \label{matrix and abelian}
Let $R$ be an abelian ring and $n \ge 2$. Then, $M_n(R)$ is GNC if, and only if, $R/J(R)$ is Boolean and $J(M_n(R))$ is nil.
\end{proposition}

\begin{proof}
$(\Leftarrow).$ Since \( R/J(R) \) is a Boolean ring, in conjunction with \cite[Corollary 6]{BCDM}, \( M_n(R)/J(M_n(R)) \cong M_n(R/J(R)) \) is a nil-clean ring. And since \( J(M_n(R)) \) is nil, Theorem \ref{factor}(1) implies that \( M_n(R) \) is a GNC ring.

$(\Rightarrow).$ Since \( M_n(R) \) is a GNC ring, in conjunction with Lemma \ref{J(R) nil}, \( J(M_n(R))=M_n(J(R)) \) is nil. Also, \( M_n(R/J(R)) \) is GNC. Therefore, \( R/J(R) \) is an exchange ring. On the other hand, we have from \cite[Corollary 2.5]{CDJ} that \( R/J(R) \) is abelian. Therefore, in view of \cite[Proposition 2.19]{chenlin}, \( R/J(R) \) is reduced and hence $2$-primal. Thus, the utilization of Proposition \ref{matrix and 2-primal} forces that \( R/J(R) \) is Boolean.
\end{proof}

\begin{corollary}
Let $R$ be a local ring and $n \ge 2$. Then, $M_n(R)$ is GNC if, and only if, $R/J(R) \cong \mathbb{Z}_2$ and $J(M_n(R))$ is nil.
\end{corollary}

\begin{proof}
$(\Leftarrow).$ Immediate.

$(\Rightarrow).$  It is enough to demonstrate only that \( R/J(R) \cong \mathbb{Z}_2 \). In fact, since \( M_n(R) \) is a GNC ring, we have that \( M_n(R/J(R)) \cong M_n(R)/J(M_n(R)) \) is also a GNC ring. And since \( R \) is local, \( R/J(R) \) is a division ring. Therefore, Lemma \ref{div GNC} insures that \( R/J(R) \cong \mathbb{Z}_2 \).
\end{proof}

The following technical claim is also of some usefulness.

\begin{lemma}\label{sum subring and nil ideal}
Let $R$ be a ring such that $R = S + K$, where $S$ is a subring of $R$ and $K$ is a nil-ideal of $R$. Then, $S$ is GNC if, and only if, $R$ is GNC.
\end{lemma}

\begin{proof}
Clearly, \( S \cap K \subseteq K \) is a nil-ideal of \( S \). Also, we routinely can write that \( R/K = (S+K)/K \cong S/(S \cap K) \). Therefore, Theorem \ref{factor} is applicable to get the desired result.
\end{proof}

Further, let $A, B$ be two rings, and let $M,N$ be the $(A,B)$-bi-module and $(B,A)$-bi-module, respectively. Also, we consider the bilinear maps $\phi : M\otimes_B N \to A$ and $\psi : N\otimes_AM \to B$ that apply to the following properties
$${\rm Id}_M \otimes_B \psi = \phi \otimes_A {\rm Id}_M, \quad {\rm Id}_N \otimes_A \phi = \psi \otimes_B {\rm Id}_N.$$
For $m \in M$ and $n \in N$, we define $mn := \phi(m \otimes n)$ and $nm := \psi(n \otimes m)$.
Thus, the $4$-tuple
$R= \begin{pmatrix}
	A & M \\
	N & B
\end{pmatrix}$
becomes to an associative ring equipped with the obvious matrix operations, which is called a {\it Morita context ring}. Denote the two-sided ideals ${\rm Im}\phi$ and ${\rm Im}\psi$ to $MN$ and $NM$, respectively, that are called the {\it trace ideals} of the Morita context.

\medskip

We now have all the ingredients needed to prove the following.

\begin{theorem}\label{morita}
Let $R=$$A~M \choose N~B$ be a Morita context such that $MN$ and $NM$ are nilpotent ideals of $A$ and $B$, respectively. Then, $R$ is GNC if, and only if, both $A$ and $B$ are nil-clean.
\end{theorem}

\begin{proof}
Apparently, since \( MN \subseteq J(A) \) and \( NM \subseteq J(B) \), by using \cite[Lemma 3.1(1)]{tangs}, we have
\( J(R) = \begin{pmatrix} J(A) & M \\ N & J(B)\end{pmatrix} \)
and \( R/J(R) \cong A/J(A) \times B/J(B) \). Since \( R \) is a GNC ring, a consultation with Theorem \ref{factor}(2) assures that that \( R/J(R) \) is also GNC. Therefore, the exploitation of Theorem \ref{factor}(3), \( A/J(A) \) and \( B/J(B) \) are nil-clean. Moreover, since \( J(R) \) is nil, we infer both \( J(A) \) and \( J(B) \) are also nil. Hence, from \cite[Proposition 3.13]{diesl}, we conclude that \( A \) and \( B \) are nil-clean.

As for the converse, let us assume that \( A \) and \( B \) are nil-clean. We have \( R = S + K \), where
\( S = \begin{pmatrix} A & 0 \\ 0 & B \end{pmatrix} \)
is a subring of \( R \) and
\( K = \begin{pmatrix} MN & M \\ N & NM \end{pmatrix} \)
is a nil-ideal of \( R \) since, by an induction on $l\geq 1$, the equality
\[ K^{2l} = \begin{pmatrix} (MN)^l & (MN)^lM \\ (NM)^lN & (NM)^l \end{pmatrix} \]
is fulfilled for every \( l \in \mathbb{N} \). Furthermore, as \( S = A \times B \), Theorem \ref{factor}(3) ensures that \( S \) is a GNC ring. Therefore, knowing Lemma \ref{sum subring and nil ideal}, we conclude that \( R \) is a GNC ring as well.
\end{proof}

Now, let $R$, $S$ be two rings, and let $M$ be an $(R,S)$-bi-module such that the operation $(rm)s = r(ms$) is valid for all $r \in R$, $m \in M$ and $s \in S$. Given such a bi-module $M$, we can set

$$
{\rm T}(R, S, M) =
\begin{pmatrix}
	R& M \\
	0& S
\end{pmatrix}
=
\left\{
\begin{pmatrix}
	r& m \\
	0& s
\end{pmatrix}
: r \in R, m \in M, s \in S
\right\},
$$
where it forms a ring with the usual matrix operations. The so-stated formal matrix ${\rm T}(R, S, M)$ is called a {\it formal triangular matrix ring}. In Theorem \ref{morita}, if we set $N =\{0\}$, then we will obtain the following two corollaries.

\begin{corollary}\label{cor3.36}
Let $R,S$ be rings and let $M$ be an $(R,S)$-bi-module. Then, the formal triangular matrix ring ${\rm T}(R,S,M)$ is GNC if, and only if, both $A$ and $B$ are nil-clean.
\end{corollary}

\begin{corollary}\label{triangular matrix}
Let $R$ be a ring and $n\geqslant 1$ is a natural number. Then, ${\rm T}_{n}(R)$ is GNC if, and only if, $R$ is nil-clean.
\end{corollary}

Given now a ring $R$ and a central element $s$ of $R$, the $4$-tuple
$\begin{pmatrix}
	R& R \\
	R& R
\end{pmatrix}$
becomes a ring with addition defined componentwise and with multiplication defined by
$$
\begin{pmatrix}
	a_1& x_1 \\
	y_1& b_1
\end{pmatrix}
\begin{pmatrix}
	a_2& x_2 \\
	y_2& b_2
\end{pmatrix}=
\begin{pmatrix}
	a_1a_2 + sx_1y_2& a_1x_2 + x_1b_2 \\
	y_1a_2 + b_1y_2& sy_1x_2 + b_1b_2
\end{pmatrix}.
$$
This ring is denoted by ${\rm K}_s(R)$. A Morita context
$
\begin{pmatrix}
	A& M \\
	N& B
\end{pmatrix}
$
with $A = B = M = N = R$ is called a {\it generalized matrix ring} over $R$. It was observed in \cite{kry} that a ring $S$ is a generalized matrix ring over $R$ if, and only if, $S = {\rm K}_s(R)$ for some $s \in {\rm Z}(R)$, the center of $R$. Here $MN = NM = sR$, so that $$MN \subseteq J(A) \Longleftrightarrow  s \in J(R), NM \subseteq J(B) \Longleftrightarrow  s \in  J(R),$$ and $MN, NM$ are nilpotent $\Longleftrightarrow  s$ is a nilpotent. Thus, Theorem \ref{morita} has the following consequence, too.

\begin{corollary}\label{generalized matrix ring}
Let $R$ be a ring and $s\in {\rm Z}(R) \cap {\rm Nil}(R)$. Then, ${\rm K}_{s}(R)$ is GNC if, and only if, $R$ is nil-clean.
\end{corollary}

Following Tang and Zhou (cf. \cite{tangac}), for $n\geq 2$ and for $s\in {\rm Z}(R)$, the $n\times n$ formal matrix ring over $R$ defined with the help of $s$, and denoted by ${\rm M}_{n}(R;s)$, is the set of all $n\times n$ matrices over $R$ with usual addition of matrices and with multiplication defined below:

\noindent For $(a_{ij})$ and $(b_{ij})$ in ${\rm M}_{n}(R;s)$, set
$$(a_{ij})(b_{ij})=(c_{ij}), \quad \text{where} ~~ (c_{ij})=\sum s^{\delta_{ikj}}a_{ik}b_{kj}.$$
Here, $\delta_{ijk}=1+\delta_{ik}-\delta_{ij}-\delta_{jk}$, where $\delta_{jk}$, $\delta_{ij}$, $\delta_{ik}$ are the standard {\it Kroncker delta} symbols.

\medskip

Thereby, we arrive at the following.

\begin{corollary}\label{cor3.39}
Let $R$ be a ring and $s\in {\rm Z}(R) \cap {\rm Nil}(R)$. Then, ${\rm M}_{n}(R;s)$ is GNC if, and only if, $R$ is nil-clean.
\end{corollary}

\begin{proof}
We shall use induction on $n$. If $n = 2$, then ${\rm M}_{2}(R;s) = Ks^2(R)$. So, the claim is true in view of Corollary \ref{generalized matrix ring}. Suppose $n > 2$ and assume that the claim holds for $M_{n-1}(R; s)$. Letting $A = M_{n-1}(R; s)$, we observe that $M_n(R; s) =
    \begin{pmatrix}A & M \\ N & R \end{pmatrix}$
is a Morita context, where $M=\begin{pmatrix} M_{1n} \\ \vdots \\ M_{n-1,n} \end{pmatrix}$ and $N = \begin{pmatrix} M_{n1 \cdots M_{n,n-1}}\end{pmatrix}$ with $M_{in} = M_{ni} = R$ for all $i = 1,..., n-1$. Moreover, for $ x= \begin{pmatrix} x_{1n} \\ \vdots \\ x_{n-1,n} \end{pmatrix}$ and $y=\begin{pmatrix} y_{n1 \cdots y_{n,n-1}}\end{pmatrix}$ , we calculate that
\begin{align}
        xy&= \begin{pmatrix}
            s^2x_{1n}y_{n1} & sx_{1n}y_{n2} & \cdots & sx_{1n}y_{n,n-1}\\
            sx_{2n}y_{n1} & s^2x_{2n}y_{n2} & \cdots & sx_{2n}y_{n,n-1} \\
            \vdots & \vdots & & \vdots \\
            sx_{n-1,n}y_{n1} & sx_{n-1,n}y_{n2} & \cdots & s^2x_{n-1,n}y_{n,n-1}
        \end{pmatrix} \in A \\
       yx&= \sum_{i=1}^{n-1}s^2y_{ni}x_{in} \in R.
\end{align}
Therefore, by (1) and (2), we have \( MN \subseteq sA \) and \( NM \subseteq s^2 R \). Moreover, since \( s \in Z(R) \cap Nil(R) \), it follows that \( MN \) and \( NM \) are nilpotent. Hence, the proof is completed with the aid of Theorem \ref{morita}, as expected.
\end{proof}

\section{GNC Group Rings}

Following the traditional terminology, we say that a group $G$ is a {\it $p$-group} if every element of $G$ is a power of the prime number $p$. Moreover, a group $G$ is said to be {\it locally finite} if every finitely generated subgroup is finite.

\medskip

Suppose now that $G$ is an arbitrary group and $R$ is an arbitrary ring. As usual, $RG$ stands for the group ring of $G$ over $R$. The homomorphism $\varepsilon :RG\rightarrow R$, defined by $\varepsilon (\displaystyle\sum_{g\in G}a_{g}g)=\displaystyle\sum_{g\in G}a_{g}$, is called the {\it augmentation map} of $RG$ and its kernel, denoted by $\Delta (RG)$, is called the {\it augmentation ideal} of $RG$.

\medskip

Before receiving our major assertion for this section, we start our considerations with the next few preliminaries.

\begin{lemma} \label{R is GNC}
If $RG$ is a GNC ring, then $R$ is too GNC.
\end{lemma}

\begin{proof}
We know that \( RG/\Delta(RG) \cong R \). Therefore, in virtue of Corollary \ref{image}, it follows that \( R \) must be a GNC ring.
\end{proof}

\begin{lemma}\label{rg}
Let $R$ be a GNC ring with $p \in Nil(R)$ and let $G$ be a locally finite $p$-group, where $p$ is a prime. Then, the group ring $RG$ is GNC.
\end{lemma}

\begin{proof}
In accordance with \cite[Proposition 16]{con}, we know that \( \Delta(RG) \) is a nil-ideal. On the other vein, it is clear that \( RG = \Delta(RG) + R \). Therefore, invoking Lemma \ref{sum subring and nil ideal}, the proof is straightforward. Alternatively, we can write the proof as follows: Since \( \Delta(RG) \) is nil and \( RG/\Delta(RG) \cong R \), Theorem \ref{factor}(1) allows us to get that \( RG \) is a GNC ring.
\end{proof}

\begin{remark} In the lemma above, the condition \( p \in Nil(R) \) is necessary and cannot be overlooked. This is because, while \( R = \mathbb{Z}_3 \) is a GNC ring and \( G = C_2 = \langle g \rangle \) is a locally finite $2$-group, the ring \( RG \) is surely {\it not} a GNC ring. It can be easily shown that \( \text{Id}(RG) = \{0, 1, 2+g, 2+2g\} \) and \( \text{Nil}(RG) = \{0\} \). Moreover, the element \( 1+g \) is a non-unit in \( RG \) that cannot be represented in the form of a nil-clean element.
\end{remark}

\begin{lemma}
Let \( R \) be a GNC ring and let \( G \) be a group such that \( \Delta(RG) \subseteq J(RG) \). Then, \( RG/J(RG) \) is a GNC ring.
\end{lemma}

\begin{proof}
Obviously, we have \( RG = \Delta(RG) + R \), because \( \Delta(RG) \subseteq J(RG) \), which leads to the fact that \( RG = J(RG) + R \). Therefore,
\[ R/(J(RG) \cap R) \cong (J(RG) + R)/J(RG) = RG/J(RG), \]
and since the left hand-side is a GNC ring, we conclude that \( RG/J(RG) \) is a GNC ring, as formulated.
\end{proof}

According to Lemma \ref{R is GNC}, if $RG$ is a GNC ring, then $R$ is also a GNC ring. In what follows, we will focus on the topic of what properties the group $G$ will have when $RG$ is a GNC ring. Explicitly, we obtain the following.

\begin{lemma}\label{char p}
Suppose $R$ is a ring of characteristic $char(R) = p$, where $p$ is a prime number, and $G$ is an abelian group. If $RG$ is a GNC ring, then $G$ is a $p$-group.
\end{lemma}

\begin{proof}
Assuming that \( 1 \neq g \in G \), then \( 1 - g \in \Delta(RG) \). So, a plain trick shows that \( 1 - g \) is not a unit. Thus, there exist \( e \in \text{Id}(RG) \) and \( q \in \text{Nil}(RG) \) such that \( 1 - g = e + q \). This enables us that \( 1 - e \in \text{Id}(RG) \cap U(RG) \), so \( e = 0 \), which leads to \( 1 - g = q \in \text{Nil}(RG) \). Hence, there exists \( k \in \mathbb{N} \) such that \( (1 - g)^{k+1} = 0 \), and from \cite[
Theorem 3.2(A)]{DL}, we conclude that \( g^{p^k} = (1 - (1 - g))^{p^k} = 1 \), as required.
\end{proof}

\begin{corollary}
Let $R$ be a ring with \( p \in \text{Nil}(R) \), where $p$ is a prime number, and $G$ is an abelian group such that $RG$ is a GNC ring. Then, $G$ is a $p$-group.
\end{corollary}

\begin{proof}
We consider the canonical surjection \( RG \to R/J(R)G \). Since \( p \in \text{Nil}(R) \), it follows that the characteristic of \( R/J(R) \) is equal to \( p \). Therefore, by Lemma \ref{char p}, the proof is complete.
\end{proof}

\begin{lemma} \label{comm}
Let $R$ be a commutative ring and $G$ be an abelian group such that $RG$ is a GNC ring. Then, $G$ is a $p$-group, where $p \in Nil(R)$.
\end{lemma}

\begin{proof}
If $RG$ is a GNC ring, seeing on Lemma \ref{abelian}, we have that either $RG$ is local with $J(RG)$ being nil, or $RG$ is a strongly nil-clean ring.

If, firstly, $RG$ is local and $J(RG)$ is nil, then owing to \cite[Corollary]{niclocal}, we deduce $\Delta(RG) \subseteq J(RG)$. Since $J(RG)$ is nil, it gives that $\Delta(RG)$ must also be nil. Therefore, looking at \cite[Proposition 16(i)]{con}, we conclude that $G$ is a $p$-group with $p \in \text{Nil}(R)$.

If now $RG$ is a strongly nil-clean ring, then according to \cite[Theorem 4.7(1)]{kosan1}, we have that $G$ is a $2$-group. Additionally, from \cite[Proposition 3.14]{diesl}, we can get $2 \in \text{Nil}(R)$, as pursued.
\end{proof}

In the following theorem, we significantly extend Lemma \ref{char p} and \ref{comm} and thus demonstrate that if $RG$ is a GNC ring, then it must hold that $\Delta(RG) \subseteq J(RG)$. Specifically, we are prepared to prove the next result.

\begin{theorem}\label{gr}
Let $R$ be a ring and $G$ an abelian group such that $RG$ is a GNC ring. Then, $G$ is a $p$-group, where $p$ belongs to ${\rm Nil}(R)$.
\end{theorem}

\begin{proof}
Let us assume $1 \neq g \in G$. Then, $1 - g \in \Delta(RG)$, which implies that $1 - g$ is not invertible. Therefore, there exist $e = e^2 \in RG$ and $q \in \text{Nil}(RG)$ such that $1 - g = e + q$. Since $G$ is abelian, we have $1 - e = g(1 - g^{-1}q) \in \text{Id}(RG) \cap \text{U}(RG)$, which guarantees that $e = 0$. Thus, $1 - g = q \in \text{Nil}(RG)$.

Besides, as $G$ is abelian, we write $1 - g \in Z(RG)$, and so $1 - g \in J(RG)$. For any $\sum a_gg \in \Delta(RG)$, we now have $\sum a_gg = \sum a_gg - \sum a_g = \sum a_g(1 - g)$. Consequently, $\Delta(RG) \subseteq J(RG)$. But, invoking \cite[Proposition 15(i)]{con}, we conclude that $G$ is a $p$-group, where $p \in J(R)$. On the other hand, from Lemma \ref{R is GNC}, we know that $R$ is a GNC ring, so that $J(R)$ is nil. Hence, $p \in J(R) \subseteq \text{Nil}(R)$, as promised.
\end{proof}

\medskip
\medskip

\noindent {\bf Funding:} The work of the first-named author, P.V. Danchev, is partially supported by the Junta de Andaluc\'ia, Grant FQM 264.

\end{document}